\documentclass[10pt,reqno]{amsart}
\usepackage{latexsym,amsmath,amssymb,amscd}
\usepackage[all]{xy}

\def\today{\ifcase \month \or
   January \or February \or March \or April \or
   May \or June \or July \or August \or
   September \or October \or November \or December \fi
   \space\number\day , \number\year}
   
\makeatletter
  \newcommand\@dotsep{4.5}
  \def\@tocline#1#2#3#4#5#6#7{\relax
     \ifnum #1>\c@tocdepth 
     \else
     \par \addpenalty\@secpenalty\addvspace{#2}%
     \begingroup \hyphenpenalty\@M
     \@ifempty{#4}{%
     \@tempdima\csname r@tocindent\number#1\endcsname\relax
        }{%
         \@tempdima#4\relax
           }%
      \parindent\z@ \leftskip#3\relax \advance\leftskip\@tempdima\relax
      \rightskip\@pnumwidth plus1em \parfillskip-\@pnumwidth
       #5\leavevmode\hskip-\@tempdima #6\relax
       \leaders\hbox{$\m@th
       \mkern \@dotsep mu\hbox{.}\mkern \@dotsep mu$}\hfill
       \hbox to\@pnumwidth{\@tocpagenum{#7}}\par
       \nobreak
        \endgroup
         \fi}
\makeatother 

\begin{document}

\makeatletter
\@addtoreset{figure}{section}
\def\thefigure{\thesection.\@arabic\c@figure}
\def\fps@figure{h,t}
\@addtoreset{table}{bsection}

\def\thetable{\thesection.\@arabic\c@table}
\def\fps@table{h, t}
\@addtoreset{equation}{section}
\def\theequation{\arabic{equation}}
\makeatother

\newcommand{\bfi}{\bfseries\itshape}

\newtheorem{theorem}{Theorem}
\newtheorem{corollary}[theorem]{Corollary}
\newtheorem{definition}[theorem]{Definition}
\newtheorem{example}[theorem]{Example}
\newtheorem{lemma}[theorem]{Lemma}
\newtheorem{notation}[theorem]{Notation}
\newtheorem{proposition}[theorem]{Proposition}
\newtheorem{remark}[theorem]{Remark}
\newtheorem{setting}[theorem]{Setting}
\newtheorem{problem}[theorem]{Problem}

\numberwithin{theorem}{section}
\numberwithin{equation}{section}

\renewcommand{\1}{{\bf 1}}
\newcommand{\Ad}{{\rm Ad}}
\newcommand{\Alg}{{\rm Alg}\,}
\newcommand{\Aut}{{\rm Aut}\,}
\newcommand{\ad}{{\rm ad}}
\newcommand{\Borel}{{\rm Borel}}
\newcommand{\botimes}{\bar{\otimes}}
\newcommand{\Cb}{{\mathcal C}_{\rm b}}
\newcommand{\Ci}{{\mathcal C}^\infty}
\newcommand{\Cpol}{{\mathcal C}^\infty_{\rm pol}}
\newcommand{\Der}{{\rm Der}\,}
\newcommand{\de}{{\rm d}}
\newcommand{\ee}{{\rm e}}
\newcommand{\End}{{\rm End}\,}
\newcommand{\ev}{{\rm ev}}
\newcommand{\hotimes}{\widehat{\otimes}}
\newcommand{\id}{{\rm id}}
\newcommand{\ie}{{\rm i}}
\newcommand{\iotaR}{\iota^{\rm R}}
\newcommand{\GL}{{\rm GL}}
\newcommand{\gl}{{{\mathfrak g}{\mathfrak l}}}
\newcommand{\Hom}{{\rm Hom}\,}
\newcommand{\Img}{{\rm Im}\,}
\newcommand{\Ind}{{\rm Ind}}
\newcommand{\Ker}{{\rm Ker}\,}
\newcommand{\Lie}{\text{\bf L}}
\newcommand{\m}{\text{\bf m}}
\newcommand{\pr}{{\rm pr}}
\newcommand{\Ran}{{\rm Ran}\,}
\renewcommand{\Re}{{\rm Re}\,}
\newcommand{\so}{\text{so}}
\newcommand{\spa}{{\rm span}}
\newcommand{\Tr}{{\rm Tr}\,}
\newcommand{\Op}{{\rm Op}}
\newcommand{\U}{{\rm U}}

\newcommand{\Ac}{{\mathcal A}}
\newcommand{\Bc}{{\mathcal B}}
\newcommand{\Cc}{{\mathcal C}}
\newcommand{\Dc}{{\mathcal D}}
\newcommand{\Ec}{{\mathcal E}}
\newcommand{\Fc}{{\mathcal F}}
\newcommand{\Hc}{{\mathcal H}}
\newcommand{\Jc}{{\mathcal J}}
\newcommand{\Lc}{{\mathcal L}}
\renewcommand{\Mc}{{\mathcal M}}
\newcommand{\Nc}{{\mathcal N}}
\newcommand{\Oc}{{\mathcal O}}
\newcommand{\Pc}{{\mathcal P}}
\newcommand{\Rc}{{\mathcal R}}
\newcommand{\Sc}{{\mathcal S}}
\newcommand{\Tc}{{\mathcal T}}
\newcommand{\Vc}{{\mathcal V}}
\newcommand{\Uc}{{\mathcal U}}
\newcommand{\Xc}{{\mathcal X}}
\newcommand{\Yc}{{\mathcal Y}}
\newcommand{\Wig}{{\mathcal W}}

\newcommand{\Bg}{{\mathfrak B}}
\newcommand{\Fg}{{\mathfrak F}}
\newcommand{\Gg}{{\mathfrak G}}
\newcommand{\Ig}{{\mathfrak I}}
\newcommand{\Jg}{{\mathfrak J}}
\newcommand{\Lg}{{\mathfrak L}}
\newcommand{\Pg}{{\mathfrak P}}
\newcommand{\Sg}{{\mathfrak S}}
\newcommand{\Xg}{{\mathfrak X}}
\newcommand{\Yg}{{\mathfrak Y}}
\newcommand{\Zg}{{\mathfrak Z}}

\newcommand{\ag}{{\mathfrak a}}
\newcommand{\bg}{{\mathfrak b}}
\newcommand{\dg}{{\mathfrak d}}
\renewcommand{\gg}{{\mathfrak g}}
\newcommand{\hg}{{\mathfrak h}}
\newcommand{\kg}{{\mathfrak k}}
\newcommand{\mg}{{\mathfrak m}}
\newcommand{\n}{{\mathfrak n}}
\newcommand{\og}{{\mathfrak o}}
\newcommand{\pg}{{\mathfrak p}}
\newcommand{\sg}{{\mathfrak s}}
\newcommand{\tg}{{\mathfrak t}}
\newcommand{\ug}{{\mathfrak u}}
\newcommand{\zg}{{\mathfrak z}}

\newcommand{\BB}{\mathbb B}
\newcommand{\CC}{\mathbb C}
\newcommand{\KK}{\mathbb K}
\newcommand{\NN}{\mathbb N}
\newcommand{\QQ}{\mathbb Q}
\newcommand{\RR}{\mathbb R}
\newcommand{\TT}{\mathbb T}
\newcommand{\ZZ}{\mathbb Z}

\newcommand{\ep}{\varepsilon}

\newcommand{\hake}[1]{\langle #1 \rangle }

\newcommand{\scalar}[2]{\langle #1 ,#2 \rangle }
\newcommand{\vect}[2]{(#1_1 ,\ldots ,#1_{#2})}
\newcommand{\norm}[1]{\Vert #1 \Vert }
\newcommand{\normrum}[2]{{\norm {#1}}_{#2}}

\newcommand{\upp}[1]{^{(#1)}}
\newcommand{\p}{\partial}

\newcommand{\opn}{\operatorname}
\newcommand{\slim}{\operatornamewithlimits{s-lim\,}}
\newcommand{\sgn}{\operatorname{sgn}}

\newcommand{\seq}[2]{#1_1 ,\dots ,#1_{#2} }
\newcommand{\loc}{_{\opn{loc}}}

\makeatletter
\title[Faithful representations of nilpotent Lie algebras]{Faithful representations of infinite-dimensional nilpotent Lie algebras}
\author{Ingrid Belti\c t\u a 
and Daniel Belti\c t\u a
}
\address{Institute of Mathematics ``Simion Stoilow'' of the Romanian Academy, 
Research Unit~1, 
P.O. Box 1-764, Bucharest, Romania}
\email{Ingrid.Beltita@imar.ro}
\email{Daniel.Beltita@imar.ro}
\date{\today}

\begin{abstract} 
For locally convex, nilpotent Lie algebras we construct 
faithful representations by nilpotent operators on a suitable locally convex space. 
In the special case of nilpotent Banach-Lie algebras we get norm continuous representations by bounded operators on Banach spaces.
\end{abstract}

\maketitle


\section{Introduction}

The aim of this short note is to provide an affirmative answer to a question raised in the seminal paper by K.-H. Neeb \cite[pag. 440, Probl. VII.2 and VIII.6(b)]{Ne06}, regarding suitable versions of Birkhoff's embedding theorem in infinite dimensions. 
In particular, we prove that every nilpotent Banach-Lie algebra has a bounded faithful representation by nilpotent operators on a suitable Banach space; 
see Corollary~\ref{main_cor} below. 

In order to put this result in a proper perspective, we recall a few classical facts: 
\begin{itemize}
\item Every finite-dimensional nilpotent Lie algebra can be realized as a Lie algebra of nilpotent matrices, by Birkhoff's embedding theorem; see \cite{Bi37}.
\item In a more general situation, every finite-dimensional Lie algebra can be faithfully represented by matrices, by Ado's theorem; see \cite{Ad35}. 
\item The above statement of Ado's theorem cannot be directly extended to the Banach category, by replacing the finite-dimensional vector spaces by Banach spaces. 
In fact, it has been long known that there exist Banach-Lie algebras that cannot be faithfully represented by Banach space operators; 
see \cite{vES73} and also \cite[pag.~438, Rem.~VIII.7]{Ne06} for a broader discussion. 
\end{itemize}
Besides this old open problem concerning the extension of Ado's theorem to infinite dimensions, there also exists a current interest in questions regarding the finite-dimensional situation. 
Thus, one has been interested in finding good estimates for the size of the matrices that have to be used in order to realize a given nilpotent Lie algebra; 
see \cite{dGr97}, \cite{BEdG09}, \cite{BM11}, and the references therein. 
Among the various methods developed in \cite{BEdG09}, 
the second one is in some sense related to the setting of the present paper, 
inasmuch as the dual of the universal enveloping algebra of any Lie algebra 
is closely related to the space of polynomial functions on that Lie algebra. 
Our point here is to provide an analytic version of that method 
which is eventually applicable to some (infinite-dimensional) topological Lie algebras, 
namely to locally convex, nilpotent Lie algebras (\cite{Li69, Ne06}). 

The main result of our paper is Theorem~\ref{main} and provides faithful representations for algebras of this type by nilpotent operators, and for the corresponding locally convex Lie groups by unipotent operators. 
In the special case of Banach-Lie algebras, we get in Corollary~\ref{main_cor} bounded representations. 

We conclude the introduction by pointing out that the aforementioned corollary 
provides a partial solution to a more general problem which was raised in \cite{Wo98} 
and is still open in its full generality:     
A quasinilpotent Banach-Lie algebra is a Banach-Lie algebra $\gg$ with the property that 
for every $x\in\gg$ the bounded linear operator $\ad_{\gg}x\colon\gg\to\gg$ is quasinilpotent, 
in the sense that its spectrum is equal to $\{0\}$ or equivalently 
$\lim\limits_{n\to\infty}\Vert(\ad_{\gg}x)^n\Vert^{1/n}=0$, 
where the operator norm is computed with respect to any norm which defines the topology of~$\gg$. 
With this terminology, it was asked in \cite[Question 2]{Wo98} 
whether every quasinilpotent Banach-Lie algebra admits a faithful continuous representation 
on some Banach space. 
It is clear that if a Banach-Lie algebra is nilpotent then it is also quasinilpotent, 
hence Corollary~\ref{main_cor} in the present paper solves in the affirmative 
the above question in the special case of the nilpotent Banach-Lie algebras.

\section{Main results}

\begin{notation}
\normalfont
We will work in the following setting: 
\begin{itemize}
\item Unless otherwise stated, $\gg$ is a Hausdorff, locally convex, nilpotent  Lie algebra over $\KK\in\{\RR,\CC\}$ 
with the nilpotency index denoted by $N+1$, where $N\in\NN$. 
Thus, if we define $\gg^{(1)}=\gg$ and $\gg^{(j)}:=[\gg,\gg^{(j-1)}]$ for $j\ge1$, 
then we have $\gg^{(N)}\ne\{0\}=\gg^{(N+1)}$. 
\item $G=(\gg,\ast)$ is the corresponding simply connected, locally convex, nilpotent  Lie group, 
whose group operation $\ast$ is defined by the Baker-Campbell-Hausdorff formula on~$\gg$ (see \cite[Ex. IV 1.6.6]{Ne06}). 
Note that in this case there exists a smooth exponential mapping, which is in fact the identity.  
Therefore we use the same letters to denote elements from $\gg$ and $G$, and  use the notation $G$ 
only when it is useful to emphasize the group structure.

\item We denote by $\Cc(\gg,\KK)$ the space of $\KK$-valued continuous functions on $\gg$
endowed with the topology of uniform convergence on the bounded sets, and by 
$$\lambda\colon G\to\End(\Cc(\gg,\KK)),\quad (\lambda(x)\phi)(y)=\phi((-x)\ast y)$$ 
the corresponding left-regular representation of $G$. 

If $\phi\in\Cc(\gg,\KK)$ and $x,y\in\gg$, then we define 
\begin{equation}\label{dot}
({ \de\lambda}(x)\phi)(y):=\lim_{t\to 0}\frac{\phi((-tx)\ast y)-\phi(y)}{t}
\end{equation} 
whenever the above limit exists. 
(Compare \cite[Def.~2.1]{BB11b}.)
\item For every locally convex space $\Yc$ over $\CC$ and every integer $m\ge0$ we denote by $\Pc_m(\gg,\Yc)$ the linear space of $\Yc$-valued, continuous polynomial functions of degree $\le m$ on~$\gg$ 
and by $\Pc^m(\gg,\Yc)$ its subspace consisting of the homogeneous polynomial functions of degree $m$ (see \cite{BS71a}). 
We endow $\Pc_m(\gg,\Yc)$ with the topology of uniform convergence on the bounded sets. 
We have a topological direct sum decomposition 
\begin{equation}\label{dot2}
\Pc_m(\gg,\Yc)
=\Pc^0(\gg,\Yc)\dotplus \Pc^1(\gg,\Yc)\dotplus\cdots\dotplus\Pc^m(\gg,\Yc),
\end{equation}
and then it easily follows by \cite[Th.~A, Th.~2]{BS71a} that $\Pc_m(\gg,\Yc)$ is a closed subspace of $\Cc(\gg,\KK)$.
\end{itemize}
\qed
\end{notation}

\begin{remark}\label{banach}
\normalfont 
Assume here that $\gg$ is a Banach-Lie algebra and let $m\in \NN$. 
The space $\Pc^m(\gg, \KK)$ is a Banach space and  one norm defining its topology can be described  as follows. 
For $\phi\in \Pc^m(\gg, \KK)$ we denote by 
$\tilde{\phi}\colon \gg\times\cdots \times \gg \to \KK$ the symmetric $m$-linear
bounded functional satisfying 
$ \phi(x)  =\tilde{\phi}(x, \dots, x)$ for every $x\in \gg$. 
Then we set $\Vert \phi\Vert_{\Pc^m(\gg, \KK)}:=\Vert \tilde{\phi}\Vert$.  (See \cite[Prop.1]{BS71a}.)
Here we use the norm 
$$ \Vert \tilde{\phi}\Vert= \sup\{\vert {\tilde\phi}(x_1,\dots, x_m) \vert \,\mid x_1, \dots x_m\in \gg, \, 
\Vert x_j\Vert\le 1, j=1, \dots m\}.$$  
The norm $\Vert \cdot \Vert_{\Pc^m(\gg, \KK)}$ is equivalent  (see \cite[Prop.1]{BS71a}) to the norm 
$$ \phi\mapsto \sup_{\Vert x\Vert\le 1} \vert \phi(x)\vert. $$
We have the following formula for the directional derivatives: 
$$(\forall x,y\in\gg)\quad \phi'_y(x)=m\widetilde{\phi}(x,y,\dots,y)$$
and 
this implies that the estimate
$$ \vert \phi'_y(x)\vert \le m \Vert x\Vert \cdot \Vert y\Vert^{m-1}  \Vert \phi\Vert_{\Pc^{m}(\gg, \KK)}
$$
for every $x$, $y \in\gg$.
\qed
\end{remark}

\begin{remark}
\normalfont
It is not clear in general that if $\phi\in\Cc(\gg,\KK)$ then the mapping $G\to\Cc(\gg,\KK)$, $x\mapsto\lambda(x)\phi$ is continuous. 
However, it is easily checked that this is the case if $\gg$ is finite-dimensional or if it is a Banach-Lie algebra and 
$\phi\in\bigcup\limits_{m\ge0}\Pc_m(\gg,\KK)$. 
It then follows by \cite[Prop.~5.1]{Ne10} that if $\Vc$ is a closed subspace 
of $\Pc_m(\gg,\KK)$ for some $m\ge0$ (hence $\Vc$ is a Banach space) 
such that $\lambda(G)\Vc\subseteq\Vc$, 
then the mapping 
$$G\times\Vc\to\Vc, \quad (x,\phi)\mapsto\lambda(x)\phi$$ 
is continuous. 
As we will see by Corollary~\ref{main_cor} below, the representation of $G$ on $\Vc$ defined by $\lambda$ is actually norm-continuous.
\qed
\end{remark}

\begin{lemma}\label{lemma1}
If $m\in\NN$, $\Phi\subseteq\Pc_m(\gg,\KK)$, and we denote 
$$\Vc_\Phi:=\overline{\spa}\,(\lambda(G)\Phi)$$
then $\Vc_\Phi$ is an invariant subspace for the regular representation~$\lambda$ and moreover we have 
$\Vc_\Phi\subseteq\Pc_{mN}(\gg,\KK)$. 
\end{lemma}

\begin{proof}
Since $\lambda$ is a group representation, 
it follows at once that $\Vc_\phi$ is an invariant subspace. 
Furthermore, let $x\in \gg$ arbitrary and define 
$$L_x\colon\gg\to\gg,\quad L_x(y)=(-x)\ast y.$$ 
Since $\gg$ is an $(N+1)$-step nilpotent Lie algebra, it follows that $L_x\in\Pc_N(\gg,\gg)$, 
and then for $\phi\in\Pc_m(\gg,\KK)$ we get $\lambda(x)\phi=\phi\circ L_x\in \Pc_{mN}(\gg,\KK)$ by \cite[Prop. 2, Cor.~1, page~63]{BS71a}. 
Thus $\lambda(G)\Phi \subseteq\Pc_{mN}(\gg,\KK)$. 
Since $\Pc_{mN}(\gg,\KK)$ is a closed subspace of $\Cc(\gg,\KK)$, we get 
$\Vc_\Phi\subseteq\Pc_{mN}(\gg,\KK)$, which concludes the proof.
\end{proof}

\begin{remark}\label{rem1}
\normalfont
If $\dim\gg=n<\infty$, then it follows by Lemma~\ref{lemma1} that every polynomial function on $\gg$ is a finite vector for the regular representation $\lambda$, in the sense that it generates a finite-dimensional cyclic subspace. 
Even more precisely, if $\phi\in\Pc_m(\gg,\KK)$ and $\Phi:=\{\phi\}$, then 
$\dim\Vc_\Phi\le\dim\Pc_{mN}(\gg,\KK)=\frac{(mN+n)!}{n!(mN)!}$.
We recall how the latter formula is obtained: 
The linear space of real polynomial functions in $n$ variables of degree at most $mN$ 
is naturally isomorphic (by using homogeneous coordinates) 
to the space of real homogeneous polynomial functions in $n+1$  variables of degree precisely $mN$, 
and the dimension of the latter space is $\frac{(mN+n)!}{n!(mN)!}$; 
see for instance \cite[Th. 1.7.5]{Me03}. 

From the point of view of the question we address here, 
the drawback of the above estimate on $\dim\Vc_\Phi$ is that it depends on the dimension of~$\gg$. 
Lemmas~\ref{corr1}, \ref{corr2}, \ref{corr3} and  \ref{corr4} below are aimed to fix this problem.   
\qed
\end{remark}

The following observation is well known, and the simple proof is included for the sake of completeness. 
We will need this result below in Lemma~\ref{corr2}, 
in order to get a uniform estimate for 
the dimension of the subalgebra generated by a finite subset of a nilpotent Lie algebra. 

\begin{lemma}\label{corr1}
Let $\mg$ be an arbitrary Lie algebra over $\KK$ and denote by $\mg_S$ the Lie subalgebra generated by some subset $S\subseteq\mg$. 
Then  
$$\mg_S=\spa_{\KK}(S\cup\{(\ad_{\mg}v_r)\cdots(\ad_{\mg}v_1)w\mid v_1,\dots,v_r,w\in S,\ r=1,2,\dots \}).$$
\end{lemma}

\begin{proof}
Let us denote the right-hand side by $\hg$. 
We have $S\subseteq\hg$ and it is clear that if some subalgebra of $\mg$ contains $S$, then it also contains $\hg$.  

Therefore, the proof will be completed as soon as we have proved that $\hg$ is a subalgebra of $\mg$. 
Since $\hg$ is by definition a linear subspace of $\mg$, 
we are left with checking that 
for every $x\in\hg$ we have $[x,\hg]\subseteq\hg$. 
To this end let us consider the normalizer of $\hg$ in $\mg$, that is, 
$\Nc_{\mg}(\hg):=\{x\in\mg\mid [x,\hg]\subseteq\mg\}$. 
By using the Jacobi identity under the form 
\begin{equation}\label{corr1_proof_eq1}
(\forall x,y\in\mg)\quad 
\ad_{\mg}([x,y])=[\ad_{\mg}x,\ad_{\mg}y]=(\ad_{\mg}x)(\ad_{\mg}y)-(\ad_{\mg}y)(\ad_{\mg}x) 
\end{equation} 
and the fact that $\hg$ is a linear subspace, it follows that if $x,y\in\Nc_{\mg}(\hg)$, 
then 
$\ad_{\mg}([x,y])\hg\subseteq\hg$, hence $[x,y]\in \Nc_{\mg}(\hg)$. 
Therefore $\Nc_{\mg}(\hg)$ is a subalgebra of $\mg$. 
On the other hand, it follows by the definition of $\hg$ that for every $x\in S$ we have $[x,\hg]\subseteq\hg$, hence the subalgebra $\Nc_{\mg}(\hg)$ contains~$S$. 
Then $\Nc_{\mg}(\hg)\supseteq\hg$ by the remark at the very beginning of the proof. 
By the definition of $\Nc_{\mg}(\hg)$ we then get $[\hg,\hg]\subseteq\hg$, hence $\hg$ is a subalgebra of $\mg$, and this completes the proof. 
\end{proof}

\begin{lemma}\label{corr2}
If $q\ge1$, then the subalgebra $\gg_{x_1,\dots,x_q}$ generated by any elements  $x_1,\dots,x_q\in\gg$ 
is finite dimensional and we have an estimate depending only on $N$ and $q$, namely  
$\dim_{\KK}(\gg_{x_1,\dots,x_q}) \le \sum\limits_{r=0}^{N-1}q^{r+1}$. 
\end{lemma}

\begin{proof}
We may assume that we have distinct elements $x_1,\dots,x_q\in\gg$, 
so that we have a subset with $q$ elements $S=\{x_1,\dots,x_q\}\subseteq\gg$. 
By using the fact that $\gg^{(N+1)}=\{0\}$ we get by Lemma~\ref{corr1} 
$$\gg_{x_1,\dots,x_q}=
\spa_{\KK}(S\cup\{(\ad_{\mg}v_r)\cdots(\ad_{\mg}v_1)w\mid v_1,\dots,v_r,w\in S,\ 
1\le r\le N-1\}).$$
Thus $\gg_{x_1,\dots,x_q}$ is linearly generated by a set containing at most $\sum\limits_{r=0}^{N-1}q^{r+1}$ elements, hence we get the required estimate.  
\end{proof}

\begin{remark}\label{finite}
\normalfont
We record the following fact for later use. 
If $x_0,x_1,\dots,x_q\in\gg$, $\alpha_1, \dots, \alpha_q\in \NN$,  and $\psi \in \Ci(\gg, \KK)$, the value 
$\bigl(({ \de\lambda}(x_1))^{\alpha_1}\cdots ({ \de\lambda}(x_q))^{\alpha_q}\psi\bigr)(x_0)$
depends only on the values of $\psi$ on the finite-dimensional subalgebra $\gg_{x_0,x_1,\dots,x_q}$ (see Lemma~\ref{corr2}). 
This is a direct consequence of the definition~\eqref{dot}. 
\qed
\end{remark}

\begin{notation}
\normalfont 
We denote by $\NN^{(\NN)}$ the set of all sequences $\alpha=(\alpha_0,\alpha_1,\dots)$ consisting of nonnegative integers such that there exists $k_\alpha\in\NN$ 
with the property $\alpha_j=0$ if $j>k_\alpha$. 
If $\alpha\in\NN^{(\NN)}$, then for any sequence $x_0,x_1,\dots\in\gg$ and any $y\in \gg$ 
we denote 
$$(\ad_{\gg}x_0)^{\alpha_0}(\ad_{\gg}x_1)^{\alpha_1}\cdots y:=
(\ad_{\gg}x_0)^{\alpha_0}(\ad_{\gg}x_1)^{\alpha_1}\cdots(\ad_{\gg}x_k)^{\alpha_k} y$$
for any $k> k_\alpha$, where $(\ad_{\gg} x)^0:=\id_{\gg}$ for every $x\in\gg$, 
and $k_\alpha\in\NN$ is as above. 

Moreover, for every $\alpha=(\alpha_0,\alpha_1,\dots)\in\NN^{(\NN)}$ 
we denote $\vert\alpha\vert=\sum\limits_{k\ge0}\alpha_k$. 
\qed
\end{notation}

\begin{lemma}\label{corr3}
Let $\phi\in\gg^*$ and $x_0\in\gg$. 
For all $\alpha,\beta\in \NN^{(\NN)}$ define 
$$p_{\alpha\beta}\colon\gg\to\KK, \quad 
p_{\alpha\beta}(y)=\phi((\ad_{\gg}x_0)^{\alpha_0}(\ad_{\gg}y)^{\beta_0}
(\ad_{\gg}x_0)^{\alpha_1}(\ad_{\gg}y)^{\beta_1}\cdots y). $$
If we denote 
$\Vc=\spa_{\KK}(\{p_{\alpha\beta}\mid \alpha,\beta\in \NN^{(\NN)}\}\cup\{\1\})$, 
then 
\begin{enumerate}
\item\label{corr3_item1}
$\dim_{\KK}\Vc\le 2^{N-1}+1$; 
\item\label{corr3_item2} 
for all $\alpha,\beta\in \NN^{(\NN)}$ we have 
$${ \de\lambda}(x_0)(p_{\alpha\beta})\in 
\spa_{\KK}(\{p_{\gamma\delta}\mid (\gamma,\delta)\in I_{\alpha\beta}\}\cup\{\1\}),$$
where $I_{\alpha\beta}$ denotes the set of all pairs 
$(\gamma,\delta)\in \NN^{(\NN)}\times \NN^{(\NN)}$ satisfying 
either 
$\vert\gamma\vert+\vert\delta\vert>\vert\alpha\vert+\vert\beta\vert$,  
 or 
$\vert\gamma\vert+\vert\delta\vert=\vert\alpha\vert+\vert\beta\vert$ 
and
$\vert\gamma\vert>\vert\alpha\vert$;  
\item\label{corr3_item3} 
${ \de\lambda}(x_0)\Vc\subseteq\Vc$ and 
$({ \de\lambda}(x_0))^{2^{N-1}+1}=0$ on $\Vc$. 
\end{enumerate}
\end{lemma}

\begin{proof}
\eqref{corr3_item1} 
It suffices to show that for at most $2^{N-1}$ pairs $(\alpha,\beta)\in \NN^{(\NN)}\times \NN^{(\NN)}$ we have $p_{\alpha\beta}\ne0$. 
For any such pair we must have $\vert\alpha\vert+\vert\beta\vert\le N-1$, since $\gg^{(N+1)}=\{0\}$.  
Moreover, if we define $k:=\max\{j\in\NN\mid \beta_j\ge1\}$, then $\alpha_{k+1}\ge1$. 
For every $r\in\{0,\dots,N-1\}$, it is then easily seen that 
the set 
$$\{(\alpha,\beta)\in \NN^{(\NN)}\times \NN^{(\NN)}\mid 
\vert\alpha\vert+\vert\beta\vert=r \text{ and }p_{\alpha\beta}\ne0\}$$
has at most $\max\{1,2^{r-1}\}$ elements. 
In fact, if $(\alpha,\beta)$ belongs to the above set then 
$p_{\alpha\beta}(y)=\phi((T_1(y)\cdots T_r(y))y)$ 
where each of $T_1(y),\dots,T_r(y)$ is either $\ad_{\gg}x_0$ or $\ad_{\gg}y$,  
and moreover $T_r(y)\ne\ad_{\gg}y$,  
and there are $\max\{1,2^{r-1}\}$ such possible choices. 
Consequently there exist at most 
$1+1+2+\cdots+2^{N-2}=2^{N-1}$ 
pairs $(\alpha,\beta)\in \NN^{(\NN)}\times \NN^{(\NN)}$ with $p_{\alpha\beta}\ne0$.

\eqref{corr3_item2} 
For every $y\in\gg$ we define $R_y\colon\gg\to\gg$, $z\mapsto z\ast y$. 
Then for every $\psi\in\Ci(\gg,\KK)$ and $y\in\gg$ we have 
$$({ \de\lambda}(x_0)\psi)(y)=(\psi\circ R_y)'_0(-x_0)=-\psi'_y((R_y)'_0 x_0). $$
It then follows by the Baker-Campbell-Hausdorff formula that there exist 
some universal constants $c_0,\dots,c_N\in\QQ$ such that 
\begin{equation}\label{corr3_proof_eq1}
(\forall \psi\in\Ci(\gg,\KK))(\forall y\in\gg)\quad
({ \de\lambda}(x_0)\psi)(y)=\sum_{j=0}^Nc_j\psi'_y((\ad_{\gg}y)^jx_0). 
\end{equation}
If $\alpha=\beta=0$, that is, for $p_{\alpha\beta}=\phi$, it follows by the above formula that 
for all $y\in\gg$ we have 
$$
({ \de\lambda}(x_0)\phi)(y)=\sum_{j=0}^Nc_j\phi((\ad_{\gg}y)^jx_0)=c_0\phi(x_0)-\sum_{j=1}^Nc_j \phi((\ad_{\gg}y)^{j-1}(\ad_{\gg}x_0)y)$$
hence ${ \de\lambda}(x_0)\phi$ is a linear combination of a constant function and of elements in the set 
$\{p_{\gamma\delta}\mid \vert\gamma\vert=1\}$. 
Thus the assertion holds true for $\alpha=\beta=0$. 

Now let $\alpha,\beta\in\NN^{(\NN)}$ with $\vert\alpha\vert+\vert\beta\vert\ge1$ and $p_{\alpha\beta}\ne0$. 
Then we may assume that there exist $\varepsilon\in\{\pm1\}$ and $k\in\NN$ such that 
$\beta_0,\alpha_1,\beta_1,\dots,\alpha_k,\beta_k\ge 1$, 
$\alpha_i=\beta_i=0$ for $i\ge k+1$, and 
$$(\forall y\in\gg)\quad 
p_{\alpha\beta}(y)=\varepsilon\phi((\ad_{\gg}x_0)^{\alpha_0}(\ad_{\gg}y)^{\beta_0}
\cdots(\ad_{\gg}x_0)^{\alpha_k}(\ad_{\gg}y)^{\beta_k}x_0) $$
(since $(\ad_{\gg}x_0)^qy=-(\ad_{\gg}x_0)^{q-1}(\ad_{\gg}y)x_0$ if $q\ge1$, 
which was already used above). 
Then for all $y,z\in\gg$ we have 
$$\begin{aligned}
(p_{\alpha\beta})'_y(z)
=\sum_{i=0}^k\sum_{l=0}^{\beta_i-1}
\varepsilon\phi(&(\ad_{\gg}x_0)^{\alpha_0}(\ad_{\gg}y)^{\beta_0}
\cdots
(\ad_{\gg}y)^l(\ad_{\gg}z)(\ad_{\gg}y)^{\beta_i-l-1}
\cdots \\
&\times (\ad_{\gg}x_0)^{\alpha_k}(\ad_{\gg}y)^{\beta_k}x_0).  
\end{aligned}$$
Therefore, by using \eqref{corr3_proof_eq1}, we get 
$$\begin{aligned}
({ \de\lambda}(x_0)(p_{\alpha\beta}))(y)
=\sum_{j=0}^N\sum_{i=0}^k\sum_{l=0}^{\beta_i-1}
c_j \varepsilon \phi(&(\ad_{\gg}x_0)^{\alpha_0}(\ad_{\gg}y)^{\beta_0}
\cdots \\
&\times (\ad_{\gg}y)^l(\ad_{\gg}((\ad_{\gg} y)^jx_0))(\ad_{\gg}y)^{\beta_i-l-1}
\cdots \\
&\times (\ad_{\gg}x_0)^{\alpha_k}(\ad_{\gg}y)^{\beta_k}x_0).  
\end{aligned}$$
On the other hand, by using the Jacobi identity \eqref{corr1_proof_eq1}, 
we see that for every $j\in\NN$ and $s\in\{0,\dots,j\}$ there exists 
$a_{js}\in\ZZ$ such that 
$$\ad_{\gg}((\ad_{\gg} y)^j x_0)=
\underbrace{[\ad_{\gg} y,\dots,[\ad_{\gg} y}_{j\text{ times}},\ad_{\gg} x_0]\cdots]
=\sum_{s=0}^j a_{js} (\ad_{\gg} y)^s(\ad_{\gg} x_0)(\ad_{\gg} y)^{j-s}.$$
This equality along with the above formula for ${ \de\lambda}(x_0)(p_{\alpha\beta})$ show that this function is a linear 
combination of the following polynomial functions of $y\in\gg$, 
for $0\le j\le N$, $0\le i\le k$, $0\le l\le\beta_i-1$, and $0\le s\le j$, 
$$\begin{aligned}
\phi(&(\ad_{\gg}x_0)^{\alpha_0}(\ad_{\gg}y)^{\beta_0}
\cdots
(\ad_{\gg}x_0)^{\alpha_i} (\ad_{\gg}y)^{l+s}(\ad_{\gg} x_0)
 (\ad_{\gg}y)^{j-s+\beta_i-l-1}\cdots  \\
&\times
 (\ad_{\gg}x_0)^{\alpha_k}(\ad_{\gg}y)^{\beta_k}x_0)
 \end{aligned}$$
 that is, 
 $$\begin{aligned}
-\phi(&(\ad_{\gg}x_0)^{\alpha_0}(\ad_{\gg}y)^{\beta_0}
\cdots
(\ad_{\gg}x_0)^{\alpha_i} (\ad_{\gg}y)^{l+s}(\ad_{\gg} x_0)
 (\ad_{\gg}y)^{j-s+\beta_i-l-1}\cdots  \\
&\times
 (\ad_{\gg}x_0)^{\alpha_k}(\ad_{\gg}y)^{\beta_k-1}(\ad_{\gg}x_0)y).
 \end{aligned}$$
The above polynomial function is equal to $-p_{\gamma\delta}$, where  
 $\vert\gamma\vert=\vert\alpha\vert+2$ and 
 $\vert\gamma\vert+\vert\delta\vert=\vert\alpha\vert+\vert\beta\vert+j$. 
Hence the assertion follows.  

\eqref{corr3_item3} 
We have already noted in the above proof of Assertion~\eqref{corr3_item1} that 
$$\sup\{\vert\gamma\vert+\vert\delta\vert\mid 
\gamma,\delta\in\NN^{(\NN)}\text{ and }p_{\gamma\delta}\ne0\}(\le N-1)<\infty$$ 
hence also 
$$\sup\{\vert\gamma\vert\mid 
\gamma\in\NN^{(\NN)}\text{ and there exists }\delta\in\NN^{(\NN)}\text{ with  }p_{\gamma\delta}\ne0\}(\le N-1)<\infty.$$
As a direct consequence of Assertion~\eqref{corr3_item2}, 
we then see that for every $\alpha,\beta\in\NN^{(\NN)}$ we have, 
besides ${ \de\lambda}(x_0)(p_{\alpha\beta})\in\Vc$, 
also $({ \de\lambda}(x_0))^m(p_{\alpha\beta})=0$ for sufficiently large $m\ge1$. 
Then ${ \de\lambda}(x_0)\Vc\subseteq\Vc$ and, since $\dim_{\KK}\Vc\le 2^{N-1}+1$ by Assertion~\eqref{corr3_item1}, we also get $({ \de\lambda}(x_0))^{2^{N-1}+1}=0$ on $\Vc$. 
This completes the proof. 
\end{proof}

\begin{lemma}\label{corr4}
If $x_0\in\gg$, then for every $m\ge0$ and $\phi\in\Pc_m(\gg,\KK)$ we have 
$$({ \de\lambda}(x_0))^{2^{N-1}m+1}\phi=0.$$ 
\end{lemma}

\begin{proof}
We have to prove that for every $y\in\gg$ we have $(({ \de\lambda}(x_0))^{2^{N-1}m+1}\phi)(y)=0$. 
By using Remark~\ref{finite}, we easily see that 
it suffices to obtain the conclusion under the additional assumption $\dim\gg<\infty$. 
Moreover, we may assume that $\phi$ is a homogeneous polynomial (of degree $m$). 
Then for every $\alpha\in\NN^{(\NN)}$ with $\vert\alpha\vert=m$ 
there exist $\phi_{\alpha,1},\dots,\phi_{\alpha,m}\in\gg^*$ such that 
$$\phi=\sum_{\stackrel{\scriptstyle \alpha\in\NN^{(\NN)}}{\vert\alpha\vert=m}} \phi_{\alpha,1}\cdots\phi_{\alpha,m}.$$
Now denote $Q:=2^{N-1}m+1$.
Since ${ \de\lambda}(x_0)\colon\Ci(\gg,\KK)\to\Ci(\gg,\KK)$ is a derivation, we get 
$$({ \de\lambda}(x_0))^{Q}\phi 
=\sum_{\stackrel{\scriptstyle \alpha\in\NN^{(\NN)}}{\vert\alpha\vert=m}}
\sum_{\stackrel{q_1,\dots,q_m\in\NN}{q_1+\cdots+q_m=Q}}
\frac{Q!}{q_1!\cdots q_m!}
({ \de\lambda}(x_0))^{q_1}\phi_{\alpha,1}\cdots({ \de\lambda}(x_0))^{q_m}\phi_{\alpha,m}.  $$
If $q_1,\dots,q_m\in\NN$ and $q_1+\cdots+q_m=Q>2^{N-1}m$, 
then there exists $i\in\{1,\dots,m\}$ such that $q_i>2^{N-1}$, 
hence $({ \de\lambda}(x_0))^{q_i}\phi_{\alpha,i}=0$ by Lemma~\ref{corr3}. 
Then $({ \de\lambda}(x_0))^{Q}\phi=0$, and this completes the proof.
\end{proof}

\begin{lemma}\label{lemma2}
If $m\in\NN$ and $\phi\in\Pc_m(\gg,\KK)$, then the mapping 
$$\gg\to\Pc_{m-1}(\gg,\KK),\quad x\mapsto \de\lambda(x)\phi,$$ 
is linear and 
for $x_1,x_2\in\gg$ we have $ \de\lambda([x_1,x_2])\phi= \de\lambda(x_1) \de\lambda(x_2)\phi- \de\lambda(x_2) \de\lambda(x_1)\phi$. 
\end{lemma}

\begin{proof}
By using Remark~\ref{finite}, 
it suffices to obtain the conclusion under the additional assumption $\dim\gg<\infty$. 
In this case we see that $\Vc_\phi:=\spa(\lambda(G)\phi)$ is a finite-dimensional invariant subspace for the regular representation $\lambda$, by using Remark~\ref{rem1}. 
Let $\de\lambda_0\colon \gg\to\End(\Vc_\phi)$, $x\mapsto\de\lambda(x)\vert_{\Vc_\phi}$, 
denote the corresponding finite-dimensional representation. 
Note that $\lambda_0\colon G\to\End(\Vc_\phi)$, $x\mapsto\lambda(x)\vert_{\Vc_\phi}$, is a finite-dimensional representation of the Lie group $G$, whose derived representation is just $\de\lambda_0\colon \gg\to\End(\Vc_\phi)$. 
Therefore, $\de\lambda_0$ is a homomorphism of Lie algebras, and we are done. 
\end{proof}

In the next statement we use a function space which is invariant to translations, and that appeared in \cite{BB11a}
in the case of finite-dimensional nilpotent Lie groups. 
See also the proof of Birkhoff's embedding theorem provided in \cite{Go76}.

\begin{theorem}\label{main}
If we define 
$$\Fc_G=\overline{\spa}\,(\lambda(G)\gg^*)$$ 
then the following assertions hold: 
\begin{enumerate}
\item\label{main_item1} 
The function space $\Fc_G$ is a closed linear subspace of $\Pc_N(\gg,\KK)$ 
which is invariant under the left regular representation 
and contains the constant functions. 
\item\label{main_item2} 
The mapping 
$\lambda_G\colon G\to\End(\Fc_G),\quad x\mapsto\lambda(x)\vert_{\Fc_G} $
is a faithful smooth representation of the Lie group $G$. 
Moreover, for every $x\in G$ we have $(\lambda_G(x)-\1)^{2^{N-1}N+1}=0$. 
\item\label{main_item3} 
The mapping 
${ \de\lambda}_G\colon \gg\to\End(\Fc_G),\quad x\mapsto\de\lambda(x)\vert_{\Fc_G} $
is a faithful representation of the Lie algebra $\gg$ and for every $x\in\gg$ we have $\de\lambda_G(x)^{2^{N-1}N+1}=0$.  
\end{enumerate}
\end{theorem}

\begin{proof} 
For Assertion~\eqref{main_item1} note that $\gg^*\subseteq\Pc_1(\gg,\KK)$, 
hence it follows by Lemma~\ref{lemma1} that 
$\Fc_G\subseteq \Pc_N(\gg,\KK)$ and $\Fc_G$ is invariant under the representation~$\lambda$. 
Moreover, if $\zg$ denotes the center of $\gg$ and we pick 
$v\in\zg$ and $\xi\in\gg^*$, 
then $\lambda(v)\xi=-\xi(v)\1+\xi$.  
We thus see that the constant functions belong to $\Fc_G$. 

We now prove Assertions \eqref{main_item2}--\eqref{main_item3}. 
Note that ${ \de\lambda}_G$ is a homomophism of Lie algebras by Lemma~\ref{lemma2}. 
The representation $\lambda_G$ is smooth 
(that is, its representation space consists only of smooth vectors) 
since every continuous polynomial function on $\gg$ has arbitrarily high G\^ateaux derivatives (see for instance \cite[Sect.~3]{BS71a}). 
Morever, Lemma~\ref{corr4} shows that for every $x\in\gg$ we have 
$(\de\lambda_G(x))^{2^{N-1}N+1}=0$. 
By considering the Taylor expansion of the function $\KK\mapsto\Fc_G$, $t\mapsto \lambda_G(tx)\phi$, 
for arbitrary $\phi\in\Fc_G$, it then easily follows that 
\begin{equation}\label{main_proof_eq1}
\lambda_G(x)=\ee^{\de\lambda_G(x)}\in\End(\Fc_G).
\end{equation} 
We then get $(\lambda_G(x)-\1)^{2^{N-1}N+1}=0$. 

In order to prove that $\lambda_G$ is a faithful representation, 
it suffices to check that if $x_0\in G$ and 
$\lambda_G(x_0)=\1\in\End(\Fc_G)$, then $x_0=0\in\gg$. 
In fact, since $\gg^*\subseteq\Fc_G$, it follows that for every $\xi\in\gg^*$ we have 
$\lambda_G(x_0)\xi=\xi$. 
By evaluating both sides of this equation at $0\in\gg$ we get $\xi(x_0)=0$ for every $\xi\in\gg^*$. 
Since the locally convex space underlying $\gg$ is a Hausdorff space, 
it then follows by the Hahn-Banach theorem that $x_0=0\in\gg$. 

Finally, since the smooth representation $\lambda_G$ is faithful and the Lie group $G$ has a smooth exponential map, it follows  that also the derived representation 
$\de\lambda_G$ is faithful. 
In fact, if $x_0\in\gg$ and $\de\lambda_G(x_0)=0$, then by 
\eqref{main_proof_eq1} we get 
$\lambda_G(x_0)=\ee^{\de\lambda_G(x_0)}=\1\in\End(\Fc_G)$. 
Therefore $x_0=0\in\gg$. 
\end{proof}

\begin{corollary}\label{main_cor}
Each connected, simply connected, $(N+1)$-step nilpotent Banach-Lie group has a faithful, unipotent, norm continuous representation on a suitable Banach space, 
with the unipotence index at most $2^{N-1}N+1$. 
Every $(N+1)$-step nilpotent Banach-Lie algebra has a faithful bounded representation by nilpotent operators on some Banach space, 
with the nilpotence index at most $2^{N-1}N+1$.
\end{corollary}

\begin{proof}
Using Theorem~\ref{main}\eqref{main_item3} in the special case when $\gg$ is a nilpotent Banach-Lie algebra, 
we are only left with proving that $\lambda_G$ is a norm continuous representation and $\de\lambda_G$ is a bounded linear map. 
It follows by Remark~\ref{banach} that the space of polynomials $\Pc_N(\gg,\KK)$ is a Banach space, 
hence so is its closed subspace $\Fc_G$. 
By using  \eqref{corr3_proof_eq1} and then Remark~\ref{banach}  we get  for suitable constants $C_1$, $C_2$ and $C_3$
\allowdisplaybreaks
\begin{align}
\Vert {\dot \lambda}(x_0) \psi \Vert_{\Pc^{m-1}(\gg,\KK)} &  \le C_1\sup_{\Vert y\Vert \le 1} 
\vert( {\dot \lambda}(x_0) \psi )(y) \vert \nonumber \\
& \le C_2 \sup_{\Vert y\Vert \le 1, 0\le j\le N} \vert \psi'_y((\ad_\gg y)^j x_0) \vert \nonumber \\
& \le  C_2 \sup_{\Vert y\Vert \le 1, 0\le j\le N} \Vert \psi\Vert_{\Pc^{m}(\gg,\KK)}\cdot
\Vert (\ad_\gg y)^j x_0\Vert\cdot  \Vert y\Vert^{m-1}  \nonumber \\
& \le C_3 \Vert \psi \Vert_{\Pc^{m}(\gg,\KK)}\Vert x_0\Vert \nonumber 
\end{align}
for every $m=1, \dots, N$.
By using the direct sum decomposition in \eqref{dot2} we
easily see that ${\dot \lambda}(x_0)$ is bounded linear operator on $\Fc_G$ and the representation
$\de\lambda_G\colon \gg\to\End(\Fc_G)$ 
is norm-continuous.  
It then integrates to a norm continuous representation of $G$, which coincides with $\lambda_G$ since $G$ is connected. 
\end{proof}

\begin{remark}
\normalfont
The second part of Corollary~\ref{main_cor} is implied by the first part by using a more general argument.  
In fact, it is easily seen that if a homomorphism of locally convex Lie groups with smooth exponential maps is injective, 
then its derivative is also injective. 
Moreover, a bounded linear operator on a Banach space is nilpotent of index $\le q$ if and only if it is the infinitesimal generator of a one-parameter group of unipotent operators with unipotence index~$\le q$.
\qed
\end{remark}



\subsection*{Acknowledgments} We wish to thank the Referee for 
several useful remarks and suggestions that helped us improve the exposition. 
This research has been partially supported by the 
Grant of the Romanian National Authority for Scientific Research, CNCS-UEFISCDI, project number PN-II-ID-PCE-2011-3-0131.

\end{document}